\documentclass[12pt,a4paper]{amsart}

\usepackage{amssymb}

\usepackage[T2A]{fontenc}
\usepackage[cp1251]{inputenc}
\usepackage[english]{babel}
\usepackage{srcltx}

\advance\textwidth20mm \advance\hoffset-10mm
\advance\textheight20mm \advance\voffset-10mm
\theoremstyle{plain}
\newtheorem{theorem}{Theorem}[section]
\newtheorem{corollary}[theorem]{Corollary}
\newtheorem{lemma}[theorem]{Lemma}
\newtheorem*{thma}{Proposition}
\theoremstyle{remark}
\newtheorem{remark}{Remark}[section]
\newtheorem*{conj}{Conjecture}

\newtheorem*{example}{Example}

\newcommand\R{\mathbb{R}}
\newcommand\N{\mathbb{N}}
\newcommand\Z{\mathbb{Z}}
\newcommand\cF{\mathcal{F}}
\newcommand\cS{\mathcal{S}}
\newcommand\cR{\mathcal{R}}
\newcommand\cH{\mathcal{H}}
\newcommand\cP{\mathcal{P}}
\newcommand\<{\langle}
\renewcommand\>{\rangle}

\newcommand\sign{\operatorname{sign}}
\newcommand\supvrai{\operatorname{supvrai}}

\begin{document}

\title[Pitt's inequalities and uncertainty principle]
{Pitt's inequalities and uncertainty principle\\ for generalized Fourier transform}

\author{D.~V.~Gorbachev}
\address{D.~Gorbachev, Tula State University,
Department of Applied Mathematics and Computer Science,
300012 Tula, Russia}
\email{dvgmail@mail.ru}

\author{V.~I.~Ivanov}
\address{V.~Ivanov, Tula State University,
Department of Applied Mathematics and Computer Science,
300012 Tula, Russia}
\email{ivaleryi@mail.ru}

\author{S.~Yu.~Tikhonov}
\address{S. Tikhonov, ICREA, Centre de Recerca Matem\`{a}tica, and UAB\\
Campus de Bellaterra, Edifici~C
08193 Bellaterra (Barcelona), Spain}
\email{stikhonov@crm.cat}

\date{\today}

\thanks{The first and the second authors were partially supported by RFFI
N\,13-01-00045, Ministry of education and science of Russian Federation
(N\,5414GZ, N\,1.1333.2014K), and Dmitry Zimin's Dynasty Foundation. The third
author was partially supported by MTM2014-59174-P, 2014 SGR 289, and RFFI
13-01-00043.}

\begin{abstract}
We study the two-parameter family of unitary operators
\[
\cF_{k,a}=\exp\Bigl(\frac{i\pi}{2a}\,(2\<k\>+{d}+a-2
)\Bigr)
\exp\Bigl(\frac{i\pi}{2a}\,\Delta_{k,a}\Bigr),
\]
which are called $(k,a)$-generalized Fourier transforms and defined by the
$a$-deformed Dunkl harmonic oscillator
$\Delta_{k,a}=|x|^{2-a}\Delta_{k}-|x|^{a}$, $a>0$, where $\Delta_{k}$ is the
Dunkl Laplacian. Particular cases of such operators are the Fourier and Dunkl
transforms. The restriction of $\cF_{k,a}$ to radial functions is given by the
$a$-deformed Hankel transform $H_{\lambda,a}$.

We obtain necessary and sufficient conditions for the weighted $(L^{p},L^{q})$
Pitt inequalities to hold for the $a$-deformed Hankel transform. Moreover, we
prove two-sided Boas--Sagher type estimates for the general monotone functions.
We~also prove sharp Pitt's inequality for $\cF_{k,a}$ transform in
$L^{2}(\R^{d})$ with the corresponding weights. Finally, we establish the
logarithmic uncertainty principle for~$\cF_{k,a}$.
\end{abstract}

\maketitle

\section{Introduction}
Let $\R^{d}$ be the real space of $d$ dimensions, equipped with a scalar
product $\<x,y\>$ and a norm $|x|=\sqrt{\<x,x\>}$. The Fourier transform is
defined by
\[
\cF(f)(y)=(2\pi)^{-d/2}\int_{\R^{d}}f(x)e^{-i\<x,y\>}\,dx.
\]
R.~Howe \cite{How87} found the spectral description of $\cF$ using the harmonic
oscillator $-(\Delta-|x|^{2})/2$ and its eigenfunctions forming the basis in
$L^{2}(\R^{d})$:
\[
\cF=\exp\Bigl(\frac{i\pi
d}{4}\Bigr)\exp\Bigl(\frac{i\pi}{4}\,(\Delta-|x|^{2})\Bigr),
\]
where $\Delta$ is the Laplace operator. This representation has been widely
used to define the fractional Fourier transform and Clifford algebra-valued
analogues, see~\cite{DeB12}.

One of the generalizations of the Fourier transform is the Dunkl transform
$\cF_{k}$ \cite{Dun92}, which is defined with the help of a root system
$R\subset \R^{d}$, a reflection group $G\subset O(d)$, and multiplicity
function $k\colon R\to \R_{+}$ such that $k$ is $G$-invariant. If~$k\equiv 0$,
we have $\cF_{k}=\cF$.

The differential-difference operator $\Delta_{k}$, the Dunkl Laplacian, plays
an important role in the Dunkl analysis, see, e.g., \cite{Ros02}. For $k\equiv
0$ we get $\Delta_{k}=\Delta$.

S.~Ben Sa\"{\i}d, T.~Kobayashi, and B.~{\O}rsted \cite{BenKobOrs12} defined
$a$-deformed Dunkl-type harmonic oscillator as follows
\[
\Delta_{k,a}=|x|^{2-a}\Delta_{k}-|x|^{a},\quad a>0.
\]
Following \cite{How87}, they constructed a two-parameter unitary operator, the
$(k,a)$-generalized Fourier transforms,
\begin{equation}\label{gen-dun-trans}
\cF_{k,a}=\exp\Bigl(\frac{i\pi}{2a}\,(2\lambda_{k}+a)\Bigr)
\exp\Bigl(\frac{i\pi}{2a}\,\Delta_{k,a}\Bigr)
\end{equation}
in $L^{2}(\R^{d},d\mu_{k,a})$ with a norm
\[
\|f\|_{2,d\mu_{k,a}}=\biggl(\int_{\R^{d}}|f(x)|^{2}\,d\mu_{k,a}(x)\biggr)^{1/2},
\]
where
\[
\lambda_{k}=\frac{d}{2}-1+\<k\>,\qquad
\<k\>=\frac{1}{2}\sum_{\alpha\in R}k(\alpha),
\]
\[
d\mu_{k,a}(x)=c_{k,a}v_{k,a}(x)\,dx,\qquad
v_{k,a}(x)=|x|^{a-2}v_{k}(x),
\]
\[
v_{k}(x)=\prod_{\alpha\in R}|\<\alpha,x\>|^{k(\alpha)},\qquad
c^{-1}_{k,a}=\int_{\R^{d}}e^{-|x|^{a}/a}v_{k,a}(x)\,dx.
\]
If $a=2$, \eqref{gen-dun-trans} recovers the Dunkl transform, and if $a=2$ and
$k\equiv 0$ the Fourier transform. For $a\ne 2$, \eqref{gen-dun-trans} is a
deformed Fourier and Dunkl operators. In particular, if $a =1$ and $k\equiv 0$,
the operator $\cF_{k,a}$ is the unitary inversion operator of the Schr\"odinger
model of the minimal representation of the group $O(N+1,2)$,
see~\cite{KobMan11}.

The operator $\cF_{k,a}$ is a unitary operator, that is, for $a>0$,
$2\<k\>+d+a>2$, it is a bijective linear operator such that for any function
$f\in L^{2}(\R^{d},d\mu_{k,a})$ the Plancherel formula holds
\cite[Th.~5.1]{BenKobOrs12}
\begin{equation}\label{planch-eq}
\bigl\|\cF_{k,a}(f)(y)\bigr\|_{2,d\mu_{k,a}}=\bigl\|f(x)\bigr\|_{2,d\mu_{k,a}}.
\end{equation}

The main goal of this paper is to prove Pitt's inequality
\begin{equation}\label{pitt-ineq}
\bigl\||y|^{-\beta}\cF_{k,a}(f)(y)\bigr\|_{2,d\mu_{k,a}}\le
C(\beta,k,a)\bigl\||x|^{\beta}f(x)\bigr\|_{2,d\mu_{k,a}},
\qquad
f\in \cS(\R^{d}),
\end{equation}
with the sharp constant
\[
C(\beta,k,a)=a^{-2\beta/a}\,
\frac{\Gamma\bigl(a^{-1}(\lambda_{k}+a/2-\beta)\bigr)}
{\Gamma\bigl(a^{-1}(\lambda_{k}+a/2+\beta)\bigr)},
\]
and the logarithmic uncertainty principle
\begin{multline}\label{unc-princ}
\int_{\R^{d}}\ln{}(|x|)|f(x)|^{2}\,d\mu_{k,a}(x)+
\int_{\R^{d}}\ln{}(|y|)|\cF_{k,a}(f)(y)|^{2}\,d\mu_{k,a}(y)\\
{}\ge
\frac{2}{a}\,\Bigl\{\psi\Bigl(\frac{\lambda_{k}}{a}+\frac{1}{2}\Bigr)+\ln a\Bigr\}
\|f\|_{2,d\mu_{k,a}}^{2},
\end{multline}
provided that
\[
0\le \beta<\lambda_{k}+\frac{a}{2},\qquad 4\lambda_{k}+a\ge 0.
\]
Here and in what follows, $\Gamma(t)$ is the gamma function,
$\psi(t)=\Gamma'(t)/\Gamma(t)$ the psi function, and $\cS(\R^{d})$ the Schwartz
space.

Inequalities \eqref{pitt-ineq} and \eqref{unc-princ} were proved by
W.~Beck\-ner \cite{Bec95} for the Fourier transform, by S.~Omri \cite{Omr11}
for the Dunkl transform on radial functions, by F.~Sol\-tani \cite{Sol14} for
the one-dimensional Dunkl transform, and by the authors \cite{GorIvaTik15} for
the general Dunkl transform. Regarding inequality \eqref{pitt-ineq} for the
Fourier transform see also \cite{Bec08, Eil01, Her77,Yaf99}.

A study of analytical properties of $\cF_{k,a}$-transform was first conducted
in \cite{BenKobOrs12}. Very recently, weighted norm inequalities were obtained
in \cite{Joh15}. In particular, the author raises the question on the sharp
logarithmic uncertainty principle for $\cF_{k,a}$.

The rest of the paper is organized as follows. In Section 2 we study the
$a$-deformed Hankel transforms which are the restriction of $\cF_{k,a}$ to
radial functions. In particular, we find necessary and sufficient conditions
for the Pitt inequalities with power weights to hold and we obtain sharp Pitt's
inequality in $L^{2}$.

Section~\ref{gm-sec} deals with boundedness properties of the $a$-deformed
Hankel transform of general monotone functions. In this case we improve the
range of parameters in the Pitt inequalities and prove the reverse
inequalities. In particular, we obtain two-sided inequalities of the
Boas--Sagher type.

Section~\ref{pitt-sec} is devoted to the proof of inequality \eqref{pitt-ineq}.
To show \eqref{pitt-ineq}, we use the following decomposition
\begin{equation}\label{expansion}
L^{2}(\R^{d},d\mu_{k,a})=\sum_{n=0}^{\infty}\oplus
\cR_{n}^{d}(v_{k,a}),\qquad \cR_{n}^{d}(v_{k,a})=\cR_{0}^{d}(v_{k,a})\otimes
\cH_{n}^{d}(v_{k}),
\end{equation}
where $\cR_{0}^{d}(v_{k,a})$~is the space of radial function, and
$\cH_{n}^{d}(v_{k})$ is the space of $k$-spherical harmonics of degree $n$.
Since $\cR_{n}^{d}(v_{k,a})$ is invariant under the operator $\cF_{k,a}$, it is
enough to study inequality \eqref{pitt-ineq} on $\cR_{n}^{d}(v_{k,a})$.

In Section~\ref{unc-princ-sec}, we obtain the logarithmic uncertainty principle
\eqref{unc-princ} for $\cF_{k,a}$-transform, which follows from
\eqref{pitt-ineq}. It is worth mentioning that the Heisenberg uncertainty
principle for $\cF_{k,a}$ was proved in \cite{BenKobOrs12}. It reads as
follows: {\itshape for $d\in \N$, $ k\ge 0$, $a > 0$, and $\,2\lambda_{k}+a>
0$, one has
\[
\bigl\||x|^{a/2}f(x)\bigr\|_{2,d\mu_{k,a}}
\bigl\||y|^{a/2}\cF_{k,a}(f)(y)\bigr\|_{2,d\mu_{k,a}}\ge
(2\lambda_{k}+a)\bigl\|f\bigr\|^{2}_{2,d\mu_{k,a}}.
\]
The equality holds if and only if the function $f$ is of the form
$f(x)=Ce^{-c|x|^{a}}$ for some $a, c>0$.} Various uncertainty relations for
$\cF_{k,a}$ were also studied in \cite{Joh15}.

We conclude by Section~\ref{last} where we study the uniform boundedness
properties of the kernel $B_{k,a}(y,x)$ in the integral transform expression
$\cF_{k,a}(f)(y)=\int_{\R^{d}}B_{k,a}(y,x)f(x)\,d\mu_{k,a}(x)$.

\section{$\cF_{k,a}$-transform on radial functions}

Let $\lambda\ge -1/2$, $J_{\lambda}(t)$ be the classical Bessel function of
degree $\lambda$, and
\[
j_{\lambda}(t)=2^{\lambda}\Gamma(\lambda+1)t^{-\lambda}J_{\lambda}(t)
\]
be the normalized Bessel function. Let also
\[
b_{\lambda}^{-1}=\int_{0}^{\infty}e^{-t^{2}/2}t^{2\lambda+1}\,dt=2^{\lambda}\Gamma(\lambda+1),\qquad
d\nu_{\lambda}(r)=b_{\lambda}r^{2\lambda+1}\,dr,\quad r\in \R_{+}.
\]
The norm in $L^{p}(\R_{+},d\nu_{\lambda})$, $1\le p<\infty$, is given by
\[
\|f\|_{p,d\nu_{\lambda}}=\biggl(\int_{\R_{+}}|f(r)|^{p}\,d\nu_{\lambda}(r)\biggr)^{1/p}.
\]
Moreover, let $\|f\|_{\infty}=\supvrai_{r\in \R_{+}}|f(r)|$.

The Hankel transform is defined as follows
\[
H_{\lambda}(f)(\rho)=\int_{\R_{+}}f(r)j_{\lambda}(\rho r)\,d\nu_{\lambda}(r).
\]
It is a unitary operator in $L^{2}(\R_{+},d\nu_{\lambda})$ and
$H_{\lambda}^{-1}=H_{\lambda}$ \cite[Chap.~7]{BatErd53}.

Note that the Hankel transform is a restriction of the Fourier transform on
radial functions if $\lambda=d/2-1$, and of the Dunkl transforms on radial
functions if $\lambda=\lambda_k=d/2-1+\<k\>$.

Let $\cS(\R_{+})$ be the Schwartz space on $\R_{+}$. For $f\in \cS(\R_{+})$, we
are interested in the Pitt inequality
\begin{equation}\label{zv}
\bigl\|\rho^{-\gamma}H_{\lambda}(f)(\rho)\bigr\|_{q,d\nu_{\lambda}}\le
c_{pq}(\beta,\gamma,\lambda)\bigl\|r^{\beta}f(r)\bigr\|_{p,d\nu_{\lambda}}
\end{equation}
with the sharp constant $c_{pq}(\beta,\gamma,\lambda)$. Here and in what
follows, we assume that $1<p\le q<\infty$.

L.~De~Carli \cite{Car08} showed that $c_{pq}(\beta,\gamma,\lambda)$ is finite
if and only if
\[
\beta-\gamma=2(\lambda+1)\Bigl(\frac{1}{p'}-\frac{1}{q}\Bigr)
\]
and
\begin{equation}\label{page3-2}
\Bigl(\frac{1}{2}-\frac{1}{p}\Bigr)(2\lambda+1)+
\max\Bigl\{\frac{1}{p'}-\frac{1}{q},0\Bigr\}\le \beta<\frac{2(\lambda+1)}{p'},
\end{equation}
where $p'$ is the H\"{o}lder conjugate of $p$.

The sharp constant $c_{pq}(\beta,\gamma,\lambda)$ is known only for $p=q=2$ and
$\gamma=\beta$ \cite{Yaf99, Omr11}:
\[
c_{22}(\beta,\beta,\lambda)=c(\beta,\lambda)=2^{-\beta}\,
\frac{\Gamma\bigl(2^{-1}(\lambda+1/2-\beta)\bigr)}
{\Gamma\bigl(2^{-1}(\lambda+1/2+\beta)\bigr)},\qquad
0\le \beta<\lambda+1.
\]

For $a>0$ we denote by $L^{p}(\R_{+},d\nu_{\lambda,a})$ the space of
complex-valued functions endowed with a norm
\[
\|f\|_{p,d\nu_{\lambda,a}}=\biggl(\int_{\R_{+}}|f(r)|^{p}\,d\nu_{\lambda,a}(r)\biggr)^{1/p},\qquad
1\le p<\infty,
\]
\[
d\nu_{\lambda,a}(r)=b_{\lambda,a}r^{2\lambda+a-1}\,dr,
\]
where the normalization constant is given by
\[
b_{\lambda,a}^{-1}=\int_{0}^{\infty}e^{-t^{a}/a}t^{2\lambda+a-1}\,dt=
a^{2\lambda/a}\Gamma\Bigl(\frac{2\lambda}{a}+1\Bigr).
\]
For $4\lambda+a\ge 0$, we define the $a$-deformed Hankel transform
\[
H_{\lambda,a}(f)(\rho)=
\int_{\R_{+}}f(r)j_{2\lambda/a}\Bigl(\frac{2}{a}\,(\rho r)^{a/2}\Bigr)\,d\nu_{\lambda,a}(r).
\]
Note that in the paper \cite[Sec. 5.5.3]{BenKobOrs12} a slightly different
definition of the $a$-deformed Hankel transform has been used (with a different
normalization). We find our definition more convenient to use.

The Hankel transform $H_{\lambda,a}$~is a unitary operator in
$L^{2}(\R_{+},d\nu_{\lambda,a})$. Moreover, if $\lambda=\lambda_{k}$, by
Bochner-type identity, the $\cF_{k,a}$ transform of a radial function is
written by the $H_{\lambda,a}$ transform (see \cite[Th. 5.21]{BenKobOrs12}):
for $f(x)=f_{0}(r)$, $r=|x|$, $\rho=|y|$, we have
\[
\cF_{k,a}(f)(y)=(\cF_{k,a}(f))_0(\rho),\qquad (\cF_{k,a}(f))_0(\rho)=H_{\lambda,a}(f_0)(\rho).
\]
Changing variables
\[
r=\Bigl(\frac{a}{2}\Bigr)^{1/a}s^{2/a},\quad
\rho=\Bigl(\frac{a}{2}\Bigr)^{1/a}\theta^{2/a},\quad
f\Bigl(\Bigl(\frac{a}{2}\Bigr)^{1/a}s^{2/a}\Bigr)=g(s),
\]
we arrive at
\begin{align*}
\int_{0}^{\infty}r^{\beta p}|f(r)|^{p}\,d\nu_{\lambda,a}(r)&=
b_{\lambda,a}\Bigl(\frac{a}{2}\Bigr)^{(\beta p+2\lambda)/a}
\int_{0}^{\infty}s^{2\beta p/a}|g(s)|^{p}s^{4\lambda/a+1}\,ds\\
&=b_{2\lambda/a}\Bigl(\frac{a}{2}\Bigr)^{\beta p/a}
\int_{0}^{\infty}s^{2\beta p/a}|g(s)|^{p}s^{4\lambda/a+1}\,ds\\
&=\Bigl(\frac{a}{2}\Bigr)^{\beta p/a}
\int_{0}^{\infty}s^{2\beta p/a}|g(s)|^{p}\,d\nu_{2\lambda/2}(s)
\end{align*}
and
\begin{multline*}
\int_{0}^{\infty}\rho^{-\gamma q}|H_{\lambda,a}(f)(\rho)|^{q}\,d\nu_{\lambda,a}(\rho)\\
=\Bigl(\frac{a}{2}\Bigr)^{-\gamma q/a}
\int_{0}^{\infty}\theta^{-2\gamma q/a}
\Bigl|H_{\lambda,a}(f)\Bigr(\Bigl(\frac{a}{2}\Bigr)^{1/a}\theta^{2/a}\Bigl)\Bigr|^{q}\,
d\nu_{2\lambda/2}(\theta),
\end{multline*}
where
\begin{align*}
H_{\lambda,a}(f)\Bigr(\Bigl(\frac{a}{2}\Bigr)^{1/a}\theta^{2/a}\Bigl)&=
b_{\lambda,a}\Bigl(\frac{a}{2}\Bigr)^{2\lambda/a}
\int_{0}^{\infty}g(s)j_{2\lambda/a}(\theta s)s^{4\lambda/a+1}\,ds\\
&=b_{2\lambda/a}\int_{0}^{\infty}g(s)j_{2\lambda/a}(\theta s)s^{4\lambda/a+1}\,ds=
H_{2\lambda/a}(g)(\theta).
\end{align*}
Therefore,
\begin{equation}\label{fgeq}
\frac{\bigl\|\rho^{-\gamma}H_{\lambda,a}(f)(\rho)\bigr\|_{q,d\nu_{\lambda,a}}}
{\bigl\|r^{\beta}f(r)\bigr\|_{p,d\nu_{\lambda,a}}}=
\Bigl(\frac{a}{2}\Bigr)^{-(\beta+\gamma)/a}\,
\frac{\bigl\|\theta^{-2\gamma/a}H_{2\lambda/a}(g)(\theta)\bigr\|_{q,d\nu_{2\lambda/a}}}
{\bigl\|s^{2\beta/a}g(s)\bigr\|_{p,d\nu_{2\lambda/a}}}.
\end{equation}
Hence, the sharp constant $c_{pq}(\beta,\gamma,\lambda,a)$ in Pitt's inequality
\begin{equation}\label{zv3}
\bigl\|\rho^{-\gamma}H_{\lambda,a}(f)(\rho)\bigr\|_{q,d\nu_{\lambda,a}}\le
c_{pq}(\beta,\gamma,\lambda,a)\bigl\|r^{\beta}f(r)\bigr\|_{p,d\nu_{\lambda,a}}
\end{equation} is related to the constant
$c_{pq}(\beta,\gamma,\lambda)$ given by \eqref{zv}
as follows
\[
c_{pq}(\beta,\gamma,\lambda,a)=\Bigl(\frac{a}{2}\Bigr)^{-(\beta+\gamma)/a}
c_{pq}\Bigl(\frac{2\beta}{a},\frac{2\gamma}{a},\frac{2\lambda}{a}\Bigr).
\]
Therefore, using the above mentioned results by De~Carli, we arrive at the following two theorems.
\begin{theorem}\label{thm-1}
Let $4\lambda+a\ge 0$ and $1< p\le q< \infty$. Pitt's inequality \eqref{zv3}
holds if and only if
\begin{align*}
&1)\quad \beta-\gamma=(2\lambda+a)\Bigl(\frac{1}{p'}-\frac{1}{q}\Bigr),\\
&2)\quad \Bigl(\frac{1}{2}-\frac{1}{p}\Bigr)\Bigl(2\lambda+\frac{a}{2}\Bigr)+
\frac{a}{2}\max\Bigl\{\frac{1}{p'}-\frac{1}{q},0\Bigr\}\le
\beta<\frac{2\lambda+a}{p'}.
\end{align*}
\end{theorem}

\begin{theorem}\label{thm-2}
Let $4\lambda+a\ge 0$ and $0\le \beta<\lambda+a/2$. Then Pitt's inequality
\[
\bigl\|\rho^{-\beta}H_{\lambda,a}(f)(\rho)\bigr\|_{2,d\nu_{\lambda,a}}\le
c(\beta,\lambda,a)\bigl\|r^{\beta}f(r)\bigr\|_{2,d\nu_{\lambda,a}}
\]
holds and the constant
\[
c(\beta,\lambda,a)=a^{-2\beta/a}\,
\frac{\Gamma\bigl(a^{-1}(\lambda+a/2-\beta)\bigr)}
{\Gamma\bigl(a^{-1}(\lambda+a/2+\beta)\bigr)}
\]
is sharp.
\end{theorem}

Let us now verify that $c(\beta,\lambda,a)$ is decreasing with $\lambda$.
\begin{lemma}\label{lem-1}
If $\alpha>0$, then
\begin{equation}\label{ineq}
\frac{\Gamma(t+\alpha)}{\Gamma(\tau+\alpha)}<\frac{\Gamma(t)}{\Gamma(\tau)}, \qquad 0<t<\tau.
\end{equation}
\end{lemma}

\begin{proof}
If $\alpha=1/2$, the proof of \eqref{ineq} can be found in \cite{Yaf99}.
To make the paper self-contained we give the proof for any $\alpha>0$.
Since the function $\psi(t)=\Gamma'(t)/\Gamma(t)$ is increasing, we have
\[
\Bigl(\frac{\Gamma(t+\alpha)}{\Gamma(t)}\Bigr)'=\frac{\Gamma(t+\alpha)}{\Gamma(t)}
\Bigl[\frac{\Gamma'(t+\alpha)}{\Gamma(t+\alpha)}-\frac{\Gamma'(t)}{\Gamma(t)}\Bigr]>0
\]
and
$\frac{\Gamma(\tau+\alpha)}{\Gamma(\tau)}>\frac{\Gamma(t+\alpha)}{\Gamma(t)}$.
\end{proof}

In this section and in what follows we use the following

\begin{remark}\label{rem-1}
Let $\cS_{0}(\R_{+})$ be a set of functions $f\in \cS(\R_{+})$ such that
$f^{(n)}(0)=0$ for any $n\in \Z_{+}$. If $f\in \cS_{0}(\R_{+})$, $\alpha\in
\R$, $\beta>0$, then $r^{\alpha}f(r^{\beta})\in \cS_{0}(\R_{+})$ and
$\cS_{0}(\R_{+})$ is dense in $L^{p}(\R_{+},r^{\alpha}\,dr)$. Therefore, when
we assume that $f\in\cS(\R_{+})$, we may additionally assume that $f\in
\cS_{0}(\R_{+})$.
\end{remark}

\section{Boas--Sagher inequalities for general monotone functions}\label{gm-sec}

In this section we study boundedness properties of the $a$-deformed Hankel
transform $H_{\lambda,a}$ of the general monotone functions. For the classical
Hankel transform
\[
H_{\lambda}(f)(\rho)=\frac{1}{2^{\lambda}\Gamma(\lambda+1)}\int_{\R_{+}}f(r)j_{\lambda}(\rho
r)r^{2\lambda+1}\,dr,
\]
similar questions were studied in \cite{CarGorTik13}.

A function $f$ of {locally} bounded variation on $[\varepsilon,\infty)$, for
any $\varepsilon>0$, is {\it general monotone}, written $f\in GM$, if it
vanishes at infinity, and there exist $C>0$ and $c>1$ such that, for every
$r>0$,
\begin{equation}\label{gm-defcond}
\int_{r}^{\infty}|df(r)|\le C\int_{r/c}^{\infty}|f(u)|\,\frac{du}{u},
\end{equation}
where $\int_a^b \,|df(u)|$ is the Riemann--Stieltjes integral.
The GM class strictly includes the
class of monotonic functions.
It was introduced in \cite{LT1} (see also \cite{LT2}).

By Theorem 1.3 from \cite{CarGorTik13} we have that {\itshape if $1<p\le
q<\infty$ and $f\in GM$, then Pitt's inequality
\begin{equation}\label{gm-pittineq}
\bigl\|\rho^{-\gamma}H_{\lambda}(f)(\rho)\bigr\|_{q,d\nu_{\lambda}}\le
c\bigl\|r^{\beta}f(r)\bigr\|_{p,d\nu_{\lambda}}
\end{equation}
holds if and only if
\begin{equation}\label{gm-bgeq}
\beta-\gamma=2(\lambda+1)\Bigl(\frac{1}{p'}-\frac{1}{q}\Bigr)
\end{equation}
and
\[
\Bigl(\frac{1}{2}-\frac{1}{p}\Bigr)(2\lambda+1)-\frac{1}{p}<\beta<\frac{2(\lambda+1)}{p'}.
\]
} It is important to note that the condition on $\beta$ is less restrictive
than the corresponding condition in the general Pitt inequality given by
\eqref{page3-2}.

Moreover, considering general monotone functions allows us to prove the reverse
Pitt inequality. First, we note that if
\begin{equation}\label{gm-intcond}
\int_{0}^{1}r^{2\lambda+1}|f(r)|\,dr+\int_{1}^{\infty}r^{\lambda+1/2}\,|df(r)|<\infty,
\end{equation}
then $H_{\lambda}(f)(\rho)$ is defined as an improper integral (i.e., as
$\lim\limits_{a\to 0,\,b\to\infty} \int_a^b$) and continuous for $\rho>0$
\cite{CarGorTik13}. The reverse Pitt inequality reads as follows: {\itshape Let
$1<q\le p<\infty$ and let a non-negative function $f\in GM$ be such that
condition \eqref{gm-intcond} is satisfied. Then the inequality
\begin{equation}\label{gm-revpittineq}
\bigl\|\rho^{-\gamma}H_{\lambda}(f)(\rho)\bigr\|_{q,d\nu_{\lambda}}\ge
c\bigl\|r^{\beta}f(r)\bigr\|_{p,d\nu_{\lambda}}
\end{equation}
holds provided that
conditions \eqref{gm-bgeq} and
\[
-\frac{2\lambda+1}{p}-\frac{1}{p}<\beta
\]
are satisfied.
}

Noting that
\[
-\frac{2\lambda+1}{p}-\frac{1}{p}<
\Bigl(\frac{1}{2}-\frac{1}{p}\Bigr)(2\lambda+1)-\frac{1}{p}=
\frac{2(\lambda+1)}{p'}-\frac{2\lambda+3}{2},
\]
inequalities \eqref{gm-pittineq} и \eqref{gm-revpittineq} imply that {\itshape if
$1<p<\infty$, $f \in GM$, $f\ge 0$ and \eqref{gm-intcond} holds, then
\begin{equation}\label{two-sided}
c_{1}\bigl\|r^{\beta}f(r)\bigr\|_{p,d\nu_{\lambda}}\le
\bigl\|\rho^{-\gamma}H_{\lambda}(f)(\rho)\bigr\|_{p,d\nu_{\lambda}}\le
c_{2}\bigl\|r^{\beta}f(r)\bigr\|_{p,d\nu_{\lambda}}
\end{equation}
if and only if \eqref{gm-bgeq} and
\[
\frac{2(\lambda+1)}{p'}-\frac{2\lambda+3}{2}<\beta<\frac{2(\lambda+1)}{p'}.
\]
} A study of two-sided inequalities of type \eqref{two-sided} for the classical
Fourier transform has a long history. In the one-dimensional case for monotone
decreasing functions the corresponding conjecture was formulated by Boas
\cite{boas}. He also obtained some partial results. Boas' conjecture was fully
solved by Sagher in \cite{sagher}. The multidimensional case was studied in
\cite{glt}.

We are going to use the above mentioned results to get direct and reverse
Pitt's inequalities for the $a$-deformed Hankel transform of the general
monotone functions. We assume that $4\lambda+a\ge 0$.

\begin{theorem}
Let $1<p\le q<\infty$ and $f\in GM$. Then Pitt's inequality
\[
\bigl\|\rho^{-\gamma}H_{\lambda,a}(f)(\rho)\bigr\|_{q,d\nu_{\lambda,a}}\le
c\bigl\|r^{\beta}f(r)\bigr\|_{p,d\nu_{\lambda,a}}
\]
holds if and only if
\begin{equation}\label{gm-a-bgeq}
\beta-\gamma=(2\lambda+a)\Bigl(\frac{1}{p'}-\frac{1}{q}\Bigr)
\end{equation}
and
\[
\Bigl(\frac{1}{2}-\frac{1}{p}\Bigr)\Bigl(2\lambda+\frac{a}{2}\Bigr)-
\frac{a}{2p}<\beta<\frac{2\lambda+a}{p'}.
\]
\end{theorem}

\begin{proof}
First we show that if $f\in GM$, then the function of the type $f(\alpha
r^{\beta})=g(r)$ is also a $GM$ function for any $\alpha,\beta>0$. Indeed,
changing variables $\alpha u^{\beta}=v$ and using inequality \eqref{gm-defcond}
for $f$, we get
\[
\int_{r}^{\infty}|dg(u)|=\int_{(r/\alpha)^{1/\beta}}^{\infty}|df(v)|\le
C\int_{(r/\alpha)^{1/\beta}/c}^{\infty}|f(v)|\,\frac{dv}{v}=
C\beta\int_{r/c}^{\infty}|g(u)|\,\frac{du}{u}.
\]
Now the proof follows from \eqref{fgeq} and the change of variables
\begin{equation}\label{gm-fgeq}
r=\Bigl(\frac{a}{2}\Bigr)^{1/a}s^{2/a},\quad
\rho=\Bigl(\frac{a}{2}\Bigr)^{1/a}\theta^{2/a},\quad
f\Bigl(\Bigl(\frac{a}{2}\Bigr)^{1/a}s^{2/a}\Bigr)=g(s).
\end{equation}
\end{proof}

\begin{theorem}
Let $1<q\le p<\infty$. Assume that $f$ is a non-negative function such that $f
\in GM$ and
\begin{equation}\label{gm-a-intcond}
\int_{0}^{1}r^{2\lambda+a-1}|f(r)|\,dr+\int_{1}^{\infty}r^{\lambda+a/4}\,|df(r)|<\infty,
\end{equation}
is satisfied.
Then the reverse Pitt inequality
\begin{equation}\label{gm-revpittineq+}
\bigl\|\rho^{-\gamma}H_{\lambda,a}(f)(\rho)\bigr\|_{q,d\nu_{\lambda,a}}\ge
c\bigl\|r^{\beta}f(r)\bigr\|_{p,d\nu_{\lambda,a}}
\end{equation}
holds provided that conditions \eqref{gm-a-bgeq} and
\[
-\frac{2\lambda+a}{p}-\frac{a}{2p}<\beta<\infty,
\]
are satisfied.
\end{theorem}

The proof follows from \eqref{fgeq} and \eqref{gm-fgeq}. Note that condition
\eqref{gm-a-intcond} implies condition \eqref{gm-intcond} for the function $g$
given by \eqref{gm-fgeq}. In particular, condition \eqref{gm-a-intcond} yields
that $H_{\lambda,a}$ is defined as an improper integral and
$H_{\lambda,a}(f)\in C(0,\infty)$; see \cite[Lemma 3.1]{CarGorTik13}.

Since
\[
-\frac{2\lambda+a}{p}-\frac{a}{2p}<
\Bigl(\frac{1}{2}-\frac{1}{p}\Bigr)\Bigl(2\lambda+\frac{a}{2}\Bigr)-
\frac{a}{2p}=\frac{2\lambda+a}{p'}-\frac{4\lambda+3a}{4},
\]
we obtain the following Boas--Sagher type equivalence.

\begin{corollary}
Suppose that $1<p<\infty$, $f \in GM$, $f\ge 0$ and \eqref{gm-a-intcond} holds.
Then
\[
c_{1}\bigl\|r^{\beta}f(r)\bigr\|_{p,d\nu_{\lambda,a}}\le
\bigl\|\rho^{-\gamma}H_{\lambda,a}(f)(\rho)\bigr\|_{p,d\nu_{\lambda,a}}\le
c_{2}\bigl\|r^{\beta}f(r)\bigr\|_{p,d\nu_{\lambda,a}}
\]
if and only if conditions
\eqref{gm-a-bgeq} and
\begin{equation}\label{gm-a-intcond-}
\frac{2\lambda+a}{p'}-\frac{4\lambda+3a}{4}<\beta<\frac{2\lambda+a}{p'}
\end{equation}
hold.
\end{corollary}

\begin{remark}
We note that condition \eqref{gm-a-intcond} always holds if
$\bigl\|r^{\beta}f(r)\bigr\|_{p,d\nu_{\lambda,a}}<\infty$ and $\beta$ satisfies
\eqref{gm-a-intcond-}. Indeed, it is easy to check (see, e.g.,
\cite[p.~111]{glt}) that any general monotone function satisfies the following
property: there is $c>1$ such that
\[
\int_{r}^{\infty}u^{\sigma}|df(u)|\le C \int_{r/c}^{\infty}u^{\sigma-1}|f(u)|\,du, \qquad
\sigma\ge 0.
\]
Then using this and H\"older's inequality, we get for $w(r)=\begin{cases}
r^{2\lambda+a-1}, & r<1, \\ r^{\lambda+a/4-1}, & r\ge 1,
\end{cases}$ that
\begin{align*}
\int_{0}^{1}r^{2\lambda+a-1}|f(r)|\,dr+\int_{1}^{\infty}r^{\lambda+a/4}\,|df(r)|
&\le C\int_{0}^{\infty}w(r)|f(r)|\,dr
\\
&\le C
\bigl\|r^{\beta}f(r)\bigr\|_{p,d\nu_{\lambda,a}}
I,
\end{align*}
where
\[
I^{p'}=\int_{0}^{\infty} r^{(-\beta -\frac{2\lambda+a-1}{p})p'} w^{p'}(r)\,dr.
\]
The latter is bounded under condition \eqref{gm-a-intcond-}.
\end{remark}

\section{Pitt's inequality for $\cF_{k,a}$ transform in $L^{2}$}\label{pitt-sec}

Recall that $R\subset \R^{d}$ is a root system, $R_{+}$ is the positive
subsystem of $R$, and $k\colon R\to \R_{+}$ is a multiplicity function with the
property that $k$ is $G$-invariant. Here~$G$ is a finite reflection group
generated by reflections $\{\sigma_{a}\colon a\in R\}$, where $\sigma_{a}$ is a
reflection with respect to a hyperplane $\<a,x\>=0$.

C.~F.~Dunkl introduced a family of first-order differential-difference operators
(Dunkl's operators) associated with $G$ and $k$ by
\[
D_{j}f(x)=\frac{\partial f(x)}{\partial x_{j}}+ \sum_{a\in
R_{+}}k(a)\<a,e_{j}\>\,\frac{f(x)-f(\sigma_{a}x)}{\<a,x\>},\quad j=1,\ldots,d.
\]
The Dunkl kernel $e_{k}(x, y)=E_{k}(x, iy)$ is the unique solution of
the joint eigenvalue problem for the corresponding Dunkl operators:
\[
D_{j}f(x)=iy_{j}f(x),\quad j=1,\ldots,d,\qquad f(0)=1.
\]
The Dunkl transform is given by
\[
\cF_{k}(f)(y)=\int_{\R^{d}}f(x)\overline{e_{k}(x,
y)}\,d\mu_{k}(x).
\]
By $\mathbb{S}^{d-1}$ denote the unit sphere in $\R^{d}$. Let $x'\in
\mathbb{S}^{d-1}$ and $dx'$ be the Lebesgue measure on the sphere. Let
\[
a_{k}^{-1}=\int_{\mathbb{S}^{d-1}}v_{k}(x')\,dx',\qquad
d\omega_{k}(x')=a_{k}v_{k}(x')\,dx',
\]
and
\[
\|f\|_{2,d\omega_{k}}=\biggl(\int_{\mathbb{S}^{d-1}}|f(x')|^{2}\,d\omega_{k}(x')\biggr)^{1/2}.
\]
For $\lambda_{k}=d/2-1+\<k\>$, we have
\begin{equation}\label{norm-const}
c_{k,a}^{-1}=
\int_{0}^{\infty}e^{-r^{a}/a}r^{2\lambda_k+a-1}\,dr\int_{\mathbb{S}^{d-1}}v_{k}(x')\,dx'=
b_{\lambda_{k},a}^{-1}a_{k}^{-1}.
\end{equation}

Let us denote by $\cH_{n}^{d}(v_{k})$ the subspace of $k$-spherical harmonics
of degree $n\in \Z_{+}$ in $L^{2} (\mathbb{S}^{d-1},d\omega_{k})$ (see
\cite[Chap.~5]{DunXu01}). Let $\cP_{n}^{d}$ be the space of homogeneous
polynomials of degree $n$ in $\R^{d}$. Then $\cH_{n}^{d}(v_{k})$ is the
restriction of $\ker \Delta_{k}\cap \cP_{n}^{d}$ to the sphere
$\mathbb{S}^{d-1}$.

If $l_{n}$ is the dimension of $\cH_{n}^{d}(v_{k})$, we denote by
$\{Y_{n}^{j}\colon j=1,\ldots,l_{n}\}$ the real-valued orthonormal basis
$\cH_{n}^{d}(v_{k})$ in $L^{2}(\mathbb{S}^{d-1},d\omega_{k})$. A union of these
bases forms an orthonormal basis in $L^{2}(\mathbb{S}^{d-1},d\omega_{k})$,
which consists of $k$-spherical harmonics, i.e., we have
\begin{equation}\label{expansion2}
L^{2}(\mathbb{S}^{d-1},d\omega_{k})=\sum_{n=0}^{\infty}\oplus
\cH_{n}^{d}(v_{k}).
\end{equation}
For $\lambda>-1$, we denote the Laguerre polynomials by
\[
L_s^{(\lambda)}(t)=\sum_{j=0}^{s}\frac{(-1)^j\Gamma(\lambda+s+1)}{(s-j)!\,\Gamma(\lambda+j+1)}\,\frac{t^{j}}{j!}.
\]
Set $\lambda_{k,a,n}=2(\lambda_k+n)/a$.
In \cite{BenKobOrs12}, the authors
constructed an orthonormal basis in $L^{2}(\R^{d},d\mu_{k,a})$
\begin{equation*}
\begin{aligned}
\Phi_{n,j,s}^{a}(x)&=\gamma_{n,j,s}^{a}
Y_{n}^{j}(x)L_s^{(\lambda_{k,a,n})}\Bigl(\frac{2}{a}\,|x|^a\Bigr)e^{-|x|^{a}/a},
\\
& \gamma_{n,j,s}^{a}>0,\quad n,s\in\Z_+,\quad j=1,\cdots, l_n,
\end{aligned}
\end{equation*}
which consists of the eigenvalues of the operator $\Delta_{k,a}=|x|^{2-a}\Delta_{k}-|x|^{a}$,
$a>0$. This helps to define
two-parameter unitary operator $\cF_{k,a}$ given by \eqref{gen-dun-trans}.

Note that the system $\{\Phi_{n,j,s}^{a}(x)\}$ is the eigensystem of $\cF_{k,a}$, i.e.,
\[
\cF_{k,a}(\Phi_{n,j,s}^{a})(y)=e^{-i \pi (s+n/a)}\Phi_{n,j,s}^{a}(y).
\]
This and \eqref{expansion2} imply the decomposition of
$L^{2}(\R^{d},d\mu_{k,a})$ given by \eqref{expansion}.

To prove Pitt's inequality, we use the following Bochner-type identity
\cite{BenKobOrs12} for functions of the type $f(x)=Y_{n}^{i}(x')\psi (r)\in
\cS(\R^{d})$, $x=rx'$:
\begin{equation}\label{boch-ident}
\begin{aligned}[t]
\cF_{k,a}(f)(y)&=e^{-i\pi n/a}\rho^nY_n^j(y')\int_{\R_{+}}\psi(r)r^{-n}j_{2(\lambda_k+n)/a}\Bigl(\frac{2}{a}\,(\rho
r)^{a/2}\Bigr)\,d\nu_{\lambda_k+n,a}(r)\\
&=e^{-i\pi n/a}\rho^nY_n^j(y')H_{\lambda_k+n,a}(\psi(r)r^{-n})(\rho),\qquad y=\rho y'.
\end{aligned}
\end{equation}
We are now in a position to prove inequality \eqref{pitt-ineq}.

\begin{theorem}\label{thm-3}
Let $\lambda_{k}=d/2-1+\<k\>$, $a>0$, $4\lambda_{k}+a\ge 0$, $0\le \beta<\lambda_{k}+a/2$.
For any $f\in \cS(\R^{d})$, the Pitt inequality
\[
\bigl\||y|^{-\beta}\cF_{k,a}(f)(y)\bigr\|_{2,d\mu_{k,a}}\le
C(\beta,k,a)\bigl\||x|^{\beta}f(x)\bigr\|_{2,d\mu_{k,a}}
\]
holds with the sharp constant
\[
C(\beta,k,a)=a^{-2\beta/a}\,
\frac{\Gamma\bigl(a^{-1}(\lambda_{k}+a/2-\beta)\bigr)}
{\Gamma\bigl(a^{-1}(\lambda_{k}+a/2+\beta)\bigr)}.
\]
\end{theorem}

\begin{proof}
For $\beta=0$ we have $C(\beta,k,a)=1$ and Pitt's inequality \eqref{pitt-ineq}
becomes Plancherel's identity \eqref{planch-eq}. The rest of the proof follows
\cite{GorIvaTik15}. Let $0<\beta<\lambda_{k}+a/2$. If $f\in
\mathcal{S}(\R^{d})$, then by \eqref{expansion2}
\[
f_{nj}(r)=\int_{\mathbb{S}^{d-1}}f(rx')Y_{n}^{j}(x')\,d\omega_{k}(x')\in \mathcal{S}(\R_{+}),
\]
\[
f(rx')=\sum_{n=0}^{\infty}\sum_{j=1}^{l_{n}}f_{nj}(r)Y_{n}^{j}(x'),
\]
\[
\int_{\mathbb{S}^{d-1}}|f(rx')|^{2}\,d\omega_{k}(x')=\sum_{n=0}^{\infty}\sum_{j=1}^{l_{n}}|f_{nj}(r)|^{2}.
\]
Using spherical coordinates, decomposition of $L^{2}(\R^{d},d\mu_{k,a})$
\eqref{expansion}, formulas \eqref{norm-const} and \eqref{boch-ident}, we get that
\begin{equation}\label{eq3.5}
\begin{aligned}[t]
\int_{\R^{d}}|x|^{2\beta}|f(x)|^{2}\,d\mu_{k,a}(x)
&=b_{\lambda_{k},a}\int_{0}^{\infty}r^{2\beta+2\lambda_k+a-1}
\int_{\mathbb{S}^{d-1}}|f(rx')|^{2}\,d\omega_{k}(x')\,dr
\\
&=b_{\lambda_{k},a}\int_{0}^{\infty}r^{2\beta+2\lambda_k+a-1}
\sum_{n=0}^{\infty}\sum_{j=1}^{l_{n}}|f_{nj}(r)|^{2}\,dr
\\
&=\sum_{n=0}^{\infty}\sum_{j=1}^{l_{n}}\int_{0}^{\infty}|f_{nj}(r)|^{2}r^{2\beta}\,d\nu_{\lambda_{k},a}(r),
\end{aligned}
\end{equation}
\[
\cF_{k,a}(f)(y)=
\sum_{n=0}^{\infty}\sum_{j=1}^{l_{n}}e^{-i\pi n/a}\rho^nY_n^j(y')H_{\lambda_k+n,a}(f_{nj}(r)r^{-n})(\rho),
\]
and
\begin{multline}\label{eq3.6}
\int_{\R^{d}}|y|^{-2\beta}|\cF_{k,a}(f)(y)|^{2}\,d\mu_{k,a}(y)
\\
\le\sum_{n=0}^{\infty}\sum_{j=1}^{l_{n}}\int_{0}^{\infty}\left|H_{\lambda_k+n,a}(f_{nj}(r)r^{-n})
(\rho)\right|^{2}\rho^{-2\beta+2n}\,d\nu_{\lambda_{k},a}(\rho).
\end{multline}
By Theorem~\ref{thm-2} with $n\in \Z_+$ and $0\le \beta<\lambda_{k}+n+a/2$, we have
\begin{multline}\label{eq3.7}
\int_{0}^{\infty}\left|H_{\lambda_k+n,a}(f_{nj}(r)r^{-n})
(\rho)\right|^{2}\rho^{-2\beta+2n}\,d\nu_{\lambda_{k},a}(\rho)
\\
\le c^{2}(\beta,\lambda_{k}+n,a)\int_{0}^{\infty}\left|f_{nj}(r)r^{-n}\right|^{2}r^{2\beta+2n}\,d\nu_{\lambda_{k},a}(r)
\\
=c^{2}(\beta,\lambda_{k}+n,a)\int_{0}^{\infty}\left|f_{nj}(r)\right|^{2}r^{2\beta}\,d\nu_{\lambda_{k},a}(r).
\end{multline}
Since $c(\beta,\lambda_{k}+n,a)$ is decreasing with $n$ (see Lemma~\ref{lem-1}), then
using \eqref{eq3.5}, \eqref{eq3.6}, and \eqref{eq3.7}, we arrive at
\begin{multline}\label{eq3.8}
\int_{\R^{d}}|y|^{-2\beta}|\cF_{k,a}(f)(y)|^{2}\,d\mu_{k,a}(y)\\
\le \sum_{n=0}^{\infty}\sum_{j=1}^{l_{n}}c^{2}(\beta,\lambda_{k}+n,a)\int_{0}^{\infty}|f_{nj}(r)|^{2}r^{2\beta}\,d\nu_{\lambda_{k},a}(r)\\
\le c^{2}(\beta,\lambda_{k},a)\sum_{n=0}^{\infty}\sum_{j=1}^{l_{n}}\int_{0}^{\infty}|f_{nj}(r)|^{2}r^{2\beta}\,d\nu_{\lambda_{k},a}(r)\\
=c^{2}(\beta,\lambda_{k},a)
\int_{\R^{d}}|x|^{2\beta}|f(x)|^{2}\,d\mu_{k,a}(x).
\end{multline}
\end{proof}

In the proof of Theorem~\ref{thm-3} we obtained the following result.

\begin{corollary}\label{c2}
Let $n\in \mathbb{N}$, $\lambda_{k}= d/2-1+\<k\>$, and
$0\le\beta<\lambda_{k}+a/2+n$. Then Pitt's inequality for the transform
$\cF_{k,a}$ holds for $f\in \mathcal{S}(\R^{d})\cap
\mathcal{R}_{n}^{d}(v_{k,a})$ with sharp constant $c(\beta,\lambda_{k}+n,a)$.
\end{corollary}

\section{Logarithmic uncertainty principle for $\cF_{k,a}$ transform}\label{unc-princ-sec}

\begin{theorem}\label{thm-4}
Suppose that $a>0$, $\lambda_{k}= d/2-1+\<k\>$, and $4\lambda_{k}+a\ge 0$. Then the inequality
\begin{multline*}
\int_{\R^{d}}\ln{}(|x|)|f(x)|^{2}\,d\mu_{k,a}(x)+
\int_{\R^{d}}\ln{}(|y|)|\cF_{k,a}(f)(y)|^{2}\,d\mu_{k,a}(y)\\
{}\ge
\frac{2}{a}\,\Bigl(
\psi\Bigl(\frac{\lambda_{k}}{a}+\frac{1}{2}\Bigr)+\ln a\Bigr)
\|f\|_{2,d\mu_{k,a}}^{2}
\end{multline*}
holds for any
$f\in \mathcal{S}(\R^{d})$.
\end{theorem}
\begin{proof}
Let us write
inequality~\eqref{pitt-ineq} in the following form
\[
\int_{\R^{d}}|y|^{-\beta}|\cF_{k,a}(f)(y)|^{2}\,d\mu_{k,a}(y)\le
c^{2}(\beta/2,
\lambda_{k},a)\int_{\R^{d}}|x|^{\beta}|f(x)|^{2}\,d\mu_{k,a}(x),
\]
where $0\le\beta<2\lambda_{k}+a$.
For $\beta\in \left(-(2\lambda_{k}+a),2\lambda_{k}+a\right)$, we define the function
\[
\varphi
(\beta)=\int_{\R^{d}}|y|^{-\beta}|\cF_{k,a}(f)(y)|^{2}\,d\mu_{k,a}(y)-
c^{2}(\beta/2,\lambda_{k},a)\int_{\R^{d}}|x|^{\beta}|f(x)|^{2}\,d\mu_{k,a}(x).
\]
Since $|\beta|<2\lambda_{k}+a$ and $f,\,\cF_{k,a}(f)\in
\mathcal{S}(\R^{d})$, then
\[
\int_{|x|\le 1}|\!\ln(|x|)||x|^{\beta}v_{k,a}(x)\,dx=
\int_{0}^{1}|\!\ln(r)|r^{\beta+2\lambda_{k}+a-1}\,dr\int_{\mathbb{S}^{d-1}}v_{k}(x')\,dx'<\infty,
\]
which implies
\[
|y|^{-\beta}\ln(|y|)|\cF_{k,a}(f)(y)|^{2}v_{k,a}(y)\in L^{1}(\R^{d})\,\,
\text{and}\,\, \ln(|x|)|x|^{\beta}|f(x)|^{2}v_{k,a}(x)\in L^{1}(\R^{d}).
\]
Hence,
\begin{multline}\label{eq4.4}
\varphi'(\beta)=-\int_{\R^{d}}|y|^{-\beta}\ln(|y|)|\cF_{k,a}(f)(y)|^{2}\,d\mu_{k,a}(y)
\\
-c^{2}(\beta/2,\lambda_{k},a)\int_{\R^{d}}|x|^{\beta}\ln(|x|)|f(x)|^{2}\,d\mu_{k,a}(x)
\\
-\frac{dc^{2}(\beta/2,\lambda_{k},a)}{d\beta}\int_{\R^{d}}|x|^{\beta}|f(x)|^{2}\,d\mu_{k,a}(x).
\end{multline}
Pitt's inequality and Plancherel's theorem imply that $\varphi (\beta)\le 0$ for $\beta>0$ and
$\varphi (0)=0$ correspondingly, hence
\[
\varphi'(0)=\lim_{\beta\to 0{+}}\frac{\varphi
(\beta)-\varphi(0)}{\beta}\le 0.
\]
Combining \eqref{eq4.4} and
\[
-\frac{dc^{2}(\beta/2,\lambda_{k},a)}{d\beta}\biggr|_{\beta=0}=
\frac{2}{a}\,\Bigl\{\psi\Bigl(\frac{\lambda_{k}}{a}+\frac{1}{2}\Bigr)+\ln a\Bigr\},
\]
we conclude the proof.
\end{proof}

\begin{remark}
In the proof of Theorem~2.1 of the paper \cite{GorIvaTik15}, sharp Pitt's
inequality in $L^{2}$ for the Hankel transform $H_\lambda$ was proved for
$\lambda>-1$. Therefore, in Theorems~\ref{thm-2}, \ref{thm-3}, and \ref{thm-4}
the conditions $4\lambda+a\ge 0$ and $4\lambda_{k}+a\ge 0$ can be replaced by
the condition $2\lambda+a>0$ and $2\lambda_{k}+a>0$ respectively.
\end{remark}

\section{Final remarks }\label{last}

The unitary operator $\cF_{k,a}$ on $L^{2}(\R^{d},d\mu_{k,a})$ can be expressed
as an integral transform \cite[(5.8)]{BenKobOrs12}
\[
\cF_{k,a}(f)(y)=\int_{\R^{d}}B_{k,a}(y,x)f(x)\,d\mu_{k,a}(x)
\]
with a symmetric kernel $B_{k,a}(x,y)$. In particular, $B_{0,2}(x,y)=e^{-i\<x,y\>}$.

A study of properties of the kernel $B_{k,a}(x,y)$ and, in particular, the conditions for
its uniform boundedness is an important problem. To illustrate,
note that if $|B_{k,a}(x,y)|\le M$, then the Hausdorff--Young inequality holds:
\[
\bigl\|\cF_{k,a}(f)\bigr\|_{p',d\mu_{k,a}}\le
M^{2/p-1}\bigl\|f\bigr\|_{p,d\mu_{k,a}},\qquad
1\le p\le 2.
\]
Therefore, it is important to know when
\begin{equation}\label{Bka-1}
|B_{k,a}(x,y)|\le B_{k,a}(0,y)=1,\qquad x,y\in \R^{d},
\end{equation}
which guaranties the accuracy of the Hausdorff--Young inequality with
constant~$1$. Moreover, one can define the generalized translation operator,
which allows to define the convolution \cite{Ros02}, the notion of modulus of
continuity \cite{IvaIva11,IvaHa15}, and different constructive and
approximation properties.

If $a=1/r$, $r\in \N$, $2\<k\>+d+a>2$, then
$\cF_{k,a}^{-1}(f)(x)=\cF_{k,a}(f)(x)$ \cite[Th.~5.3]{BenKobOrs12}, i.e., for
$f\in L^{2}(\R^{d},d\mu_{k,a})$,
\[
f(x)=\int_{\R^{d}}B_{k,a}(x,y)\cF_{k,a}(f)(y)\,d\mu_{k,a}(y).
\]
If condition \eqref{Bka-1} holds, the generalized translation operator is defined by
\[
T^{t}(f)(x)=\int_{\R^{d}}B_{k,a}(t,y)B_{k,a}(x,y)\cF_{k,a}(f)(y)\,d\mu_{k,a}(y),\qquad
t\in \R^{d}.
\]
Similarly, if $a=2/(2r+1)$, $r\in \Z_{+}$, $2\<k\>+d+a>2$, then
$\cF_{k,a}^{-1}(f)(x)=\cF_{k,a}(f)(-x)$ \cite[Th.~5.3]{BenKobOrs12}, that is,
\[
f(x)=\int_{\R^{d}}B_{k,a}(-x,y)\cF_{k,a}(f)(y)\,d\mu_{k,a}(y)
\]
and
\[
T^{t}(f)(x)=\int_{\R^{d}}B_{k,a}(-t,y)B_{k,a}(-x,y)\cF_{k,a}(f)(y)\,d\mu_{k,a}(y),\qquad
t\in \R^{d}.
\]
The operators act in $L^{2}(\R^{d},d\mu_{k,a})$ and $\|T^{t}\|=1$.

If $d=1$, using \cite[Sect.~5.4]{BenKobOrs12}, we can define
\[
B_{k,a}^\text{even}(x,y)=\frac{1}{2}\,\bigl[B_{k,a}(x,y)+B_{k,a}(x,-y)\bigr].
\]
Then
\[
B_{k,a}^\text{even}(x,y)=j_{(2k-1)/a}\bigl(\frac{2}{a}\,|xy|^{a/2}\bigr).
\]
For $2k+1+a>2$ we have $(2k-1)/a>-1$. The inequality
$|B_{k,a}^\text{even}(x,y)|\le 1$ holds only when $(2k-1)/a\ge -1/2$ or,
equivalently, $2k+a/2\ge 1$. In this case the generalized translation operator
can be defined by the formula
\[
T^{t}(f)(x)=\int_{\R^{d}}B_{k,a}^\text{even}(t,y)B_{k,a}(\pm x,y)\cF_{k,a}(f)(y)\,d\mu_{k,a}(y),\qquad
t\in \R^{1}.
\]

\begin{thma}
Assume that
\[
2\<k\>+d+a>2.
\]
Inequality \eqref{Bka-1} may not be true in general.
\end{thma}
Cf. \cite[Th.~5.11]{BenKobOrs12} and
\cite[L. 2.13]{Joh15}.
To prove Proposition, we construct the following

\begin{example} Let $d=1$, $\<k\>=k\ge 0$, and
$2\<k\>>1-a$.
First,
we remark that the kernel $B_{k,a}$ can be given by
\cite[(5.18)]{BenKobOrs12}
\[
B_{k,a}(x,y)=j_{(2k-1)/a}\Bigl(\frac{2}{a}\,|xy|^{a/2}\Bigr)+
\frac{\Gamma((2k-1)/a+1)}{\Gamma((2k+1)/a+1)}\,\frac{xy}{(ai)^{2/a}}\,
j_{(2k+1)/a}\Bigl(\frac{2}{a}\,|xy|^{a/2}\Bigr).
\]
Therefore, if $a=1$ and $k>0$, we get
\[
B_{k,1}(x,y)=j_{2k-1}(t)-\sign{}(xy)\,\frac{(t/2)^{2}}{2k(2k+1)}\,j_{2k+1}(t),\qquad
t=2|xy|^{1/2}.
\]
Let us now investigate when $|B_{k,1}(x,y)|\le 1$, $x,y\in \R$, for different values of $k$.

Taking into account the known properties of the Bessel function
\begin{align*}
&J_{\nu-1}(t)+J_{\nu+1}(t)=2\nu t^{-1}J_{\nu}(t),\\
&J_{\nu-1}(t)-J_{\nu+1}(t)=2J_{\nu}'(t),
\end{align*}
we get for $\nu=2k$ that
\[
B_{k,1}(x,y)=\begin{cases}
j_{2k}(t),&xy\le 0,\\
2^{2k}\Gamma(2k)t^{1-2k}J_{2k}'(t),&xy\ge 0.
\end{cases}
\]
Hence, $|B_{k,1}(x,y)|\le 1$ for $xy\le 0$ and for any $k\ge 0$.

\medbreak
\underline{Case $0<k<1/4$}. Using asymptotic formula for the derivative of the
Bessel function
\[
J_{\nu}'(t)=-\sqrt{\frac{2}{\pi t}}\,\bigl\{\sin{}(t-\nu
\pi/2-\pi/4)+O(t^{-1})\bigr\},\qquad t\to +\infty,
\]
we obtain that the kernel $B_{k,1}(x,y)$ is not bounded for $0<k<1/4$, $xy>0$,
and is uniformly bounded for $k\ge1/4$, $x,y\in \R$.

\medbreak
\underline{Case $k=1/4$}. Using $J_{1/2}(t)=(\pi t/2)^{-1/2}\sin t$, we get for
$xy>0$
\[
B_{1/4,1}(x,y)=2(\cos t-t^{-1}\sin t)
\]
and then $\max_{x,y\in \R} |B_{1/4,1}(x,y)|\approx 2.13$.

\medbreak
\underline{Case $1/4<k<1/2$}. Easy computer calculations show that
$|B_{k,1}(x,y)|\le M_{k}$ for $x,y\in \R_{+}$, where $M_{k}=\max_{x,y\in
\R_{+}}|B_{k,1}(x,y)|>1$ for $k\in (1/4,k_{0})$ and $M_{k}=1$ for $k\in
[k_{0},1/2)$. Moreover, $k_{0}\approx 0.44$. The number $k_{0}$ can be found
from the condition that the first minimum of the function
$2^{2k}\Gamma(2k)t^{1-2k}J_{2k}'(t)$ for $t>0$ is equal to $-1$.

\medbreak
\underline{Case $k\ge 1/2$}. For the kernel $B_{k,1}$, the following integral
representation with a nonnegative weight holds: \cite[(5.17),
(5.19)]{BenKobOrs12}
\[
B_{k,1}(x,y)=\frac{\Gamma(k+1/2)}{\Gamma(k)\Gamma(1/2)}\int_{-1}^{1}
j_{k-1}\bigl(\sqrt{2|xy|(1+\sign{}(xy)u)}\,\bigr)(1+u)(1-u^2)^{k-1}\,du.
\]
Since $|j_{\lambda}(t)|\le 1$ for $t\in \R$ and $\lambda\ge -1/2$, then
$|B_{k,1}(x,y)|\le B_{k,1}(0,0)=1$, $x,y\in \R$ for any $k\ge 1/2$.

\end{example}

We formulate the following

\begin{conj}
Inequality \eqref{Bka-1} holds whenever $2\<k\>+d+a\ge3$.
\end{conj}
In particular, we expect that if $d\ge 3$, then inequality \eqref{Bka-1} always holds.
Calculations above for the case $d=1$ and results of the paper
\cite{DeB12} for $d=2$ show that
the condition $2\<k\>+d+a\ge3$ is only sufficient for \eqref{Bka-1} to hold.

\end{document}